\documentclass[12pt,leqno]{amsart}

\usepackage{amscd}
\usepackage{amsmath}
\usepackage{amsfonts}
\usepackage{amssymb}
\usepackage{amsthm}
\usepackage{url}

\theoremstyle{plain}
\newtheorem{thm}{Theorem}[section]
\newtheorem{lem}[thm]{Lemma}

\newtheorem{cor}[thm]{Corollary}

\theoremstyle{definition}

\theoremstyle{remark}

\title{Bogomolov multipliers for unitriangular groups}
\author{Ivo M. Michailov}
\address{Faculty of Mathematics and Informatics, Shumen University "Episkop Konstantin Preslavski", Universitetska str. 115, 9700 Shumen, Bulgaria}
\email{ivo\_michailov@yahoo.com}
\date{\today}
\thanks{This work is partially supported by a project No RD-08-241/12.03.2013 of Shumen University}
\begin{document}
\baselineskip 20pt
\newcommand{\Gal}{{\rm Gal}}
\newcommand{\im}{{\rm im}}
\newcommand{\res}{{\rm res}}
\newcommand{\GL}{{\rm GL}}
\newcommand{\Br}{{\rm Br}}
\newcommand{\lcm}{{\rm lcm}}
\newcommand{\ord}{{\rm ord}}
\newcommand{\ut}{\text{UT}_n(\mathbb F_{p})}
\newcommand{\f}{\mathbb F_{p}}
\renewcommand{\thefootnote}{\fnsymbol{footnote}}
\numberwithin{equation}{section}
\begin{abstract}
The Bogomolov multiplier $B_0(G)$ of a finite group $G$ is defined
as the subgroup of the Schur multiplier consisting of the cohomology
classes vanishing after restriction to all abelian subgroups of $G$.
In this paper we give a positive answer to an open problem posed by
Kang and Kunyavski\u{i} in \cite{KK}. Namely, we prove that if $G$
is either a unitriangular group over $\f$, a quotient of its lower
central series, a subgroup of its lower central series, or a central
product of two unitriangular groups, then $B_0(G)=0$.
\end{abstract}

\maketitle

\section{Introduction}
\label{1}

Let $K$ be a field, $G$ a finite group and $V$ a faithful
representation of $G$ over $K$. Then there is a natural action of
$G$ upon the field of rational functions $K(V)$. The rationality
problem then asks whether the field of $G$-invariant functions
$K(V)^G$ is rational (i.e., purely transcendental) over $K$. A
question related to the above mentioned is whether $K(V)^G$ is
stably rational, that is, whether there exist independent variables
$x_1,\dots,x_r$ such that $K(V)^G(x_1,\dots,x_r)$ becomes a purely
transcendental extension of $K$. This problem has close connection
with L\"uroth's problem \cite{Sh} and the inverse Galois problem
\cite{Sa,Sw}.

Saltman \cite{Sa} found examples of groups $G$ of order $p^9$ such
that $\mathbb C(V)^G$ is not stably rational over $\mathbb C$. His
main method was application of the unramified cohomology group
$H_{nr}^2(\mathbb C(V)^G,\mathbb Q/\mathbb Z)$ as an obstruction.
Bogomolov \cite{Bo} proved that $H_{nr}^2(\mathbb C(V)^G,\mathbb
Q/\mathbb Z)$ is canonically isomorphic to
\begin{equation*}
B_0(G)=\bigcap_{A}\ker\{\res_G^A:H^2(G,\mathbb Q/\mathbb Z)\to
H^2(A,\mathbb Q/\mathbb Z)\}
\end{equation*}
where $A$ runs over all the bicyclic subgroups of $G$ (a group $A$
is called bicyclic if $A$ is either a cyclic group or a direct
product of two cyclic groups). The group $B_0(G)$ is a subgroup of
the Schur multiplier $H^2(G,\mathbb Q/\mathbb Z)$, and
Kunyavski\u{i} \cite{Ku} called it the \emph{Bogomolov multiplier}
of $G$. Thus the vanishing of the Bogomolov multiplier is an
obstruction to Noether's problem.

Recently, Moravec \cite{Mo1} used a notion of the nonabelian
exterior square $G\wedge G$ of a given group $G$ to obtain a new
description of $B_0(G)$. Namely, he proved that $B_0(G)$ is
(non-canonically) isomorphic to the quotient group $M(G)/M_0(G)$,
where $M(G)$ is the kernel of the commutator homomorphism $G\wedge
G\to [G,G]$, and $M_0(G)$ is the subgroup of $M(G)$ generated by all
$x\wedge y$ such that $x,y\in G$ commute.

The Bogomolov multipliers for the groups of order $p^n$ for $n\leq
6$ were calculated recently in \cite{HoK,HKK,Mo2,CM}.  Kang and
Kunyavski\u{i} \cite{KK} raised the question whether the Bogomolov
multiplier of any unitriangular group is trivial. In our main
result, Theorem \ref{main}, we prove that if $G$ is either a
unitriangular group over $\f$, a quotient of its lower central
series, or a subgroup of its lower central series, then $B_0(G)=0$.
The key idea to prove our result is to show that the multiplication
rules in unitriangular groups allow us to group the same commutators
into powers modulo $M_0(G)$. We use a technique similar to
introducing a partial order of commutators. For example, if it is
given an ordered set $a_1,a_2,\dots,a_k$ such that $[a_i,a_j]=a_k$
for $k>i,j$, or $[a_i,a_j]=1$, we can always write arbitrary product
of $a_i$'s as $\prod_{i=1}^ka_i^{n(i)}$. After that we apply several
times Lemma \ref{l1} to show that any element from $M(G)$ is also in
$M_0(G)$.

Another open problem mentioned in \cite{KK} is whether the Bogomolov
multiplier of a central product of two groups is trivial, where both
groups have trivial  Bogomolov multipliers.

Let $G$ be a central product of two groups $G_1$ and $G_2$ with a
common central subgroup. Let $\theta:K_1\to K_2$ be an isomorphism,
where $K_1\leq Z(G_1)$ and $K_2\leq Z(G_2)$, and let $E=G_1\times
G_2$. Then the central product of $G_1$ and $G_2$ is defined as the
quotient group $G=E/N$, where $N=\{ab: a\in K_1,b\in K_2,
\theta(a)=b^{-1}\}\in Z(E)$. Recently, Michailov \cite{Mi}
established the following.

\begin{thm}\label{cp}
{\rm (Michailov}\cite[Theorem 3.1]{Mi}{\rm )} Let $\theta:G_1\to
G_2$ be a group homomorphism such that its restriction
$\theta\vert_{K_1}:K_1\to K_2$ is an isomorphism, where $K_1\leq
Z(G_1)$ and $K_2\leq Z(G_2)$. Let $G$ be a central product of $G_1$
and $G_2$, i.e., $G=E/N$, where $E=G_1\times G_2$ and $N=\{ab: a\in
K_1,b\in K_2, \theta(a)=b^{-1}\}$. If
$B_0(G_1/K_1)=B_0(G_1)=B_0(G_2)=0$ then $B_0(G)=0$.
\end{thm}

With the help of Theorem \ref{cp}, we prove in Corollary \ref{cor}
that the Bogomolov multiplier of a central product of two
unitriangular groups is also trivial.

\section{Preliminaries and notations}
\label{2}

Let $G$ be a group and $x,y\in G$. We define $x^y=y^{-1}xy$ and
write $[x,y]=x^{-1}x^y=x^{-1}y^{-1}xy$ for the commutator of $x$ and
$y$. We define the commutators of higher weight as
$[x_1,x_2,\dots,x_n]=[[x_1,\dots,x_{n-1}],x_n]$ for
$x_1,x_2,\dots,x_n\in G$.

The nonabelian exterior square of $G$ is a group generated by the
symbols $x\wedge y$ ($x,y\in G$), subject to the relations
{\allowdisplaybreaks\begin{eqnarray*} &xy\wedge z&=(x^y\wedge
z^y)(y\wedge z),\\
&x\wedge yz&=(x\wedge z)(x^z\wedge y^z),\\
&x\wedge x&=1,
\end{eqnarray*}}
for all $x,y,z\in G$. We denote this group by $G\wedge G$. Let
$[G,G]$ be the commutator subgroup of $G$. Obverse that the
commutator map $\kappa : G\wedge G\to [G,G]$, given by $x\wedge
y\mapsto [x,y]$ is a well-defined group homomorphism. Let $M(G)$
denote the kernel of $\kappa$, and $M_0(G)$ detone the subgroup of
$M(G)$ generated by all $x\wedge y$ such that $x,y\in G$ commute.
Moravec proved in \cite{Mo1} that $B_0(G)$ is (non-canonically)
isomorphic to the quotient group $M(G)/M_0(G)$.

There is also an alternative way to obtain the non-abelian exterior
square $G\wedge G$. Let $\varphi$ be an automorphism of $G$ and
$G^\varphi$ be an isomorphic copy of $G$ via $\varphi : x\mapsto
x^\varphi$. We define $\tau(G)$ to be the group generated by $G$ and
$G^\varphi$, subject to the following relations:
$[x,y^\varphi]^z=[x^z,(y^z)^\varphi]=[x,y^\varphi]^{z^\varphi}$ and
$[x,x^\varphi]=1$ for all $x, y, z\in G$. Obviously, the groups $G$
and $G^\varphi$ can be viewed as subgroups of $\tau(G)$. Let
$[G,G^\varphi]=\langle[x,y^\varphi] : x,y\in G\rangle$ be the
commutator subgroup. Notice that the map $\phi : G\wedge G\to
[G,G^\varphi]$ given by $x\wedge y\mapsto [x,y^\varphi]$ is actually
an isomorphism of groups (see \cite{BM}).

Now, let $\kappa^*=\kappa\cdot\phi^{-1}$ be the composite map from
$[G,G^\varphi]$ to $[G,G]$, $M^*(G)=\ker\kappa^*$ and
$M_0^*(G)=\phi(M_0(G))$. Then $B_0(G)$ is clearly isomorphic to
$M^*(G)/M_0^*(G)$ by \cite{Mo1}. Notice that
\begin{equation*}
M^*(G)=\left\{\prod_{\text{finite}}[x_i,y_i^\varphi]^{\varepsilon_i}\in
[G,G^\varphi] : \varepsilon_i=\pm 1,
\prod_{\text{finite}}[x_i,y_i]^{\varepsilon_i}=1\right\},
\end{equation*}
and
\begin{equation*}
M_0^*(G)=\left\{\prod_{\text{finite}}[x_i,y_i^\varphi]^{\varepsilon_i}\in
[G,G^\varphi] : \varepsilon_i=\pm 1, [x_i,y_i]=1\right\}.
\end{equation*}
In order to prove that $B_0(G)=0$ for a given group $G$, it suffices
to show that $M^*(G)=M_0^*(G)$. This can be achieved by finding a
generating set of $M^*(G)$ consisting solely of elements of
$M_0^*(G)$.

The advantage of the above description of $G\wedge G$ is the ability
of using the full power of the commutator calculus instead of
computing with elements of $G\wedge G$. The following Lemma collects
various properties of $\tau(G)$ and $[G,G^\varphi]$ that will be
used in the proofs of our main results.

\begin{lem}\label{l1}
{\rm (}\cite{BM}{\rm)} Let $G$ be a group.
\begin{enumerate}
    \item $[x,yz]=[x,z][x,y][x,y,z]$ and $[xy,z]=[x,z][x,z,y][y,z]$ for all $x,y,z\in G$.
    \item If $G$ is nilpotent of class $c$, then $\tau(G)$ is nilpotent of class at most
    $c+1$.
    \item If $G$ is nilpotent of class $\leq 2$, then $[G,G^\varphi]$ is abelian.
    \item $[x,y^\varphi]=[x^\varphi,y]$ for all $x,y\in G$.
    \item $[x,y,z^\varphi]=[x,y^\varphi,z]=[x^\varphi,y,z]=[x^\varphi,y^\varphi,z]=[x^\varphi,y,z^\varphi]=[x,y^\varphi,z^\varphi]$ for all $x,y,z\in G$.
    \item $[[x,y^\varphi],[a,b^\varphi]]=[[x,y],[a,b]^\varphi]$ for all $x,y,a,b\in G$.
    \item $[x^n,y^\varphi]=[x,y^\varphi]^n=[x,(y^\varphi)^n]$ for all integers $n$ and $x,y\in G$ with $[x,y]=1$.
    \item If $[G,G]$ is nilpotent of class $c$, then $[G,G^\varphi]$ is nilpotent of class $c$ or $c+1$.
\end{enumerate}
\end{lem}

\section{Unitriangular groups}
\label{3}

Let $\ut$ denote the unitriangular group consisting of all $n\times
n$ upper-unitriangular matrices having entries in a finite field
$\f$, where $n\geq 2$. It is well known that this group is generated
by elementary transvections $t_{ij}(\lambda)$, where $1\leq i<j\leq
n$ and $\lambda\in\f^*$, which satisfy the following relations (see
\cite[\S 3]{KM}):
\begin{equation}\label{e3.1}
[t_{ik}(\alpha),t_{mj}(\beta)]=\left\{
\begin{array}{ll}
t_{ij}(\alpha\beta), &\ \text{if}\ k=m,\\
t_{mk}(-\alpha\beta), &\ \text{if}\ i=j,\\
\mathbf 1_n, &\ {\rm if}\ k\ne m\ \text{and}\ i\ne j,
\end{array}\right.
\end{equation}
where $\mathbf 1_n$ is the unity matrix of degree $n$, $1\leq
i,m\leq n-1$ and $2\leq k,j\leq n$. We assume also that
$t_{ij}(0)=\mathbf 1_n$.

Notice that the group $\ut$ is generated by $t_{i,i+1}(\f)$ for
$1\leq i\leq n-1$, where $t_{i,i+1}(\f)$ is the subgroup of $\ut$
generated by $\{t_{i,i+1}(\alpha)\mid\alpha\in\f\}$. Thus
\begin{equation*}
\ut=\langle t_{1,2}(\f),t_{2,3}(\f),\dots,t_{n-1,n}(\f)\rangle.
\end{equation*}

Let $\text{UT}_n^\ell(\f)$ be the subgroup of all matrices from
$\ut$ having $\ell-1$ zero diagonals above the main diagonal. We
have the chain of subgroups
\begin{equation*}
\ut=\text{UT}_n^1(\f)\geq \text{UT}_n^2(\f)\geq\cdots\geq
\text{UT}_n^{n-2}(\f),
\end{equation*}
where $\text{UT}_n^\ell(\f)=\langle t_{ij}(\f)\mid j-i\geq
\ell\rangle$.

Applying \eqref{e3.1}, we obtain the formula
\begin{equation}\label{e3.2}
[\text{UT}_n^r(\f),\text{UT}_n^s(\f)]=\text{UT}_n^{r+s}(\f).
\end{equation}

Now, recall that the subgroups in the lower central series of $\ut$
are defined inductively by
\begin{equation*}
\gamma_1(\ut)=\ut\quad\text{and}\quad\gamma_{\ell+1}(\ut)=[\gamma_\ell(\ut),\ut].
\end{equation*}
From \eqref{e3.2} it follows that
$\gamma_\ell(\ut)=\text{UT}_n^\ell(\f)$ for $1\leq \ell\leq n-2$,
and in particular $\gamma_{n-1}(\ut)=\{1\}$.

Next, let $\Gamma_{n,\ell}=\ut/\gamma_\ell(\ut)$ where $1\leq
\ell\leq n-2$. Obviously, $\Gamma_{n,1}=\{1\}$ and
$\Gamma_{n,2}\simeq \f \underbrace{\oplus\dots\oplus}_{n-1}\f$.

We are now going to prove the main result of this paper.

\begin{thm}\label{main}
Let the group $G$ be isomorphic to any of the groups
$\ut,\text{UT}_n^\ell(\f)$ or $\Gamma_{n,\ell}$, where $n\geq 2$ and
$1\leq \ell\leq n-2$. Then $B_0(G)=0$.
\end{thm}
\begin{proof}
{\bf Case I.} Let $G=\ut$ for arbitrary $n\geq 2$. It is easy to
verify that $t_{ij}(\alpha)t_{ij}(\beta)=t_{ij}(\alpha+\beta)$ for
any $\alpha,\beta\in\f$ and $1\leq i<j\leq n$. Therefore
$t_{ij}(\f)=\langle t_{ij}(1)\rangle$ is isomorphic to the additive
group of $\f$. We will write henceforth $t_{ij}$ instead of
$t_{ij}(1)$. Then the group $[G,G^\varphi]$ is generated modulo
$M_0^*(G)$ by the commutators $[t_{ij},t_{jk}^\varphi]$ for $1\leq
i<j<k\leq n$. We divide the proof into three steps.

{\it Step 1.} We are going to prove by induction that
$[t_{jk}^\varphi,t_{ij}^a]\equiv
[t_{jk}^\varphi,t_{ij}]^a\pmod{M_0^*(G)}$ for any $a\in\f$. In
particular it follows that $[t_{jk}^\varphi,t_{ij}]^p\in M_0^*(G)$

Assume that our claim is true for $a-1$. From Lemma \ref{l1} it
follows that {\allowdisplaybreaks\begin{eqnarray*}
&[t_{jk}^\varphi,t_{ij}^a]&= [t_{jk}^\varphi,t_{ij}][t_{jk}^\varphi,t_{ij}^{a-1}][t_{jk}^\varphi,t_{ij}^{a-1},t_{ij}]\\
&&\equiv [t_{jk}^\varphi,t_{ij}]^a[[t_{jk},t_{ij}^{a-1}],t_{ij}^\varphi]=[t_{jk}^\varphi,t_{ij}]^a[t_{ik}(-a+1),t_{ij}^\varphi],\\
&&\equiv [t_{jk}^\varphi,t_{ij}]^a\pmod{M_0^*(G)},
\end{eqnarray*}}
where $[t_{ik}(-a+1),t_{ij}^\varphi]\in M_0^*(G)$.

{\it Step 2.} We are going to show that every element $w\in
[G,G^\varphi]$ can be written as
\begin{equation}\label{e3.3}
w=\prod_{k=2}^{n-1}\prod_{i=1}^{n-k}\prod_{j=i+1}^{i+k-1}[t_{i,j},t_{j,i+k}^\varphi]^{n(i,j,k)}\cdot\tilde
w,
\end{equation}
where $n(i,j,k)\in\f$ and $\tilde w\in M_0^*(G)$.

According to Lemma \ref{l1}, we have the relation
\begin{equation}\label{e3.4}
[[t_{ij},t_{jv}^\varphi],[t_{rs},t_{su}^\varphi]]=[[t_{ij},t_{jv}],[t_{rs},t_{su}]^\varphi]=[t_{iv},t_{ru}^\varphi],
\end{equation}
where $1\leq i<j<v\leq n$ and $1\leq r<s<u\leq n$. If $v\ne r$ and
$i\ne u$ then $[t_{iv},t_{ru}^\varphi]\in M_0^*(G)$. If $v=r$ or
$i=u$ then $[t_{iv},t_{ru}^\varphi]$ is a generator modulo
$M_0^*(G)$.


Applying formula \eqref{e3.4} for $i=1,j=2,k=3$, we can write
$w=[t_{12},t_{23}^\varphi]^{n(1,2,2)}w_1$, where
$[t_{12},t_{23}^\varphi]$ does not appear as a multiplier in the
decomposition of $w_1\in [G,G^\varphi]$. Since
$[t_{12},t_{23}^\varphi]^p\in M_0^*(G)$ (see Step 1), we may assume
that $n(1,2,2)\in\f$. Next, we can write
$w=[t_{12},t_{23}^\varphi]^{n(1,2,2)}[t_{23},t_{34}^\varphi]^{n(2,3,2)}w_2$,
where $[t_{12},t_{23}^\varphi]$ and $[t_{23},t_{34}^\varphi]$ do not
appear as multipliers in the decomposition of $w_2\in
[G,G^\varphi]$. In this way we get the decomposition
\begin{equation*}
w=\prod_{i=1}^{n-2}[t_{i,i+1},t_{i+1,i+2}^\varphi]^{n(i,i+1,2)}\cdot
w_{n-2},
\end{equation*}
where $n(i,i+1,2)\in\f$ and $w_{n-2}\in [G,G^\varphi]$ does not
contain any of the commutators in the product.

Similarly, we can write
\begin{equation*}
w_{n-2}=\prod_{i=1}^{n-3}\prod_{j=i+1}^{i+2}[t_{i,j},t_{j,i+3}^\varphi]^{n(i,j,3)}\cdot
w_{2(n-3)},
\end{equation*}
where $w_{2(n-3)}\in [G,G^\varphi]$ does not contain any of the
commutators in the product.

We can proceed by induction on $k$ to obtain the final decomposition
\eqref{e3.3}. We have shown how to construct the products for $k=2$
and $k=3$. Suppose that we have the desired decomposition for $k-1$.
When we apply the commutation rule \eqref{e3.4} for any commutator
$[t_{i,j},t_{j,i+k}^\varphi]$, there might appear a commutator
$[t_{i,i+k},t_{i+k,u}^\varphi]$ which, however, does not appear in
$$\prod_{m=2}^{k}\prod_{i=1}^{n-m}\prod_{j=i+1}^{i+m-1}[t_{i,j},t_{j,i+m}^\varphi]^{n(i,j,m)}$$
(clearly, $k\geq m>j-i$). This fact allows us to group together the
same commutators into powers modulo $M_0^*(G)$.

{\it Step 3.} From \eqref{e3.1} it follows that
\begin{equation*}
w^{\kappa^*}=\prod_{k=2}^{n-1}\prod_{i=1}^{n-k}t_{i,i+k}^{m(i,k)},
\end{equation*}
where $m(i,k)=\sum_{j=i+1}^{i+k-1}n(i,j,k)$. Recall that
$t_{i,i+k}$'s are independent generators of $G$, i.e., any product
of their powers is equal to $1$ (i.e., the unitary matrix $\mathbf
1_n$) if and only if each power is equal to $1$. Therefore, $w\in
M^*(G)$ if and only if $m(i,k)=0\in\f$ for all $i,k$.

Note that $m(i,2)=n(i,i+1,2)$. From Step 1 it follows that
$[t_{i,i+1},t_{i+1,i+2}^\varphi]^{n(i,i+1,2)}\in M_0^*(G)$ for
$1\leq i\leq n-2$.

Now, choose arbitrary $i$ and $k\geq 3$. Then
$n(i,i+k-1,k)=-\sum_{j=i+1}^{i+k-2}n(i,j,k)\in\f$. Since the
commutators
$[t_{i,i+1},t_{i+1,i+3}^\varphi],\dots,[t_{i,i+k-1},t_{i+k-1,i+k}^\varphi]$
commute pairwise modulo $M_0^*(G)$, and
$[t_{i,i+k-1},t_{i+k-1,i+k}^\varphi]^{p}\in M_0^*(G)$ (see Step 1) ,
we have that
\begin{equation*}
\prod_{j=i+1}^{i+k-1}[t_{i,j},t_{j,i+k}^\varphi]^{n(i,j,k)}\equiv\prod_{j=i+1}^{i+k-2}([t_{i,j},t_{j,i+k}^\varphi][t_{i,i+k-1},t_{i+k-1,i+k}^\varphi]^{-1})^{n(i,j,k)}\pmod{M_0^*(G)}.
\end{equation*}

We are going to show that
$[t_{i,j},t_{j,i+k}^\varphi][t_{i,i+k-1},t_{i+k-1,i+k}^\varphi]^{-1}$
belongs to $M_0^*(G)$ for any $j$. From this it will easily follow
that $\prod_{j=i+1}^{i+k-1}[t_{i,j},t_{j,i+k}^\varphi]^{n(i,j,k)}$
belongs to $M_0^*(G)$, and from \eqref{e3.3} it will follow that $w$
belongs to $M_0^*(G)$.

For abuse of notation put $r=i+k-1,s=i+k$. Note that
{\allowdisplaybreaks\begin{eqnarray*}
&&[t_{ij}t_{rs},t_{js}t_{ir}]=[t_{ij},t_{ir}]^{t_{rs}}[t_{ij},t_{js}]^{t_{ir}t_{rs}}[t_{rs},t_{ir}][t_{rs},t_{js}]^{t_{ir}}=t_{is}^{t_{ir}t_{rs}}t_{is}^{-1}=1,
\end{eqnarray*}}
hence $[t_{ij}t_{rs},(t_{js}t_{ir})^\varphi]\in M_0^*(G)$. Expanding
the latter using the class restriction and Lemma \ref{l1}, we get
{\allowdisplaybreaks\begin{eqnarray*}
&[t_{ij}t_{rs},(t_{js}t_{ir})^\varphi]&=[t_{ij},t_{ir}^\varphi]^{t_{rs}}[t_{ij},t_{js}^\varphi]^{t_{ir}^\varphi t_{rs}}[t_{rs},t_{ir}^\varphi][t_{rs},t_{js}^\varphi]^{t_{ir}^\varphi}\\
&&=[t_{ij},t_{ir}^\varphi]^{t_{rs}}[t_{ij},t_{js}^\varphi][t_{ij},t_{js}^\varphi,t_{ir}^\varphi
t_{rs}][t_{rs},t_{ir}^\varphi][t_{rs},t_{js}^\varphi]^{t_{ir}^\varphi}.
\end{eqnarray*}}
Observe that $[t_{ij},t_{ir}^\varphi]^{t_{rs}},
[t_{ij},t_{js}^\varphi,t_{ir}^\varphi t_{rs}]$ and
$[t_{rs},t_{js}^\varphi]^{t_{ir}^\varphi}$ all belong to $M_0^*(G)$.
Thus we conclude that
$[t_{ij},t_{js}^\varphi][t_{ir},t_{rs}^\varphi]^{-1}=[t_{ij},t_{js}^\varphi][t_{rs},t_{ir}^\varphi]\in
M_0^*(G)$, as required.

{\bf Case II.} Let $G=\text{UT}_n^\ell(\f)$ for arbitrary $\ell\geq
2$ and $n\geq 2$. We will again write $t_{ij}$ instead of
$t_{ij}(1)$. Then the group $[G,G^\varphi]$ is generated modulo
$M_0^*(G)$ by the commutators $[t_{ij},t_{jk}^\varphi]$ for
$j-i\geq\ell$ and $k-j\geq\ell$. Note that if $2\ell\geq n$, then
according to \eqref{e3.1} the group $G$ is abelian, so $B_0(G)=0$.

Assume that $2\ell\leq n-1$. Following Step 2 of Case I, we can show
that every element $w\in [G,G^\varphi]$ can be written as
\begin{equation*}
w=\prod_{k=\ell+1}^{n-\ell}\prod_{i=1}^{n-k-\ell+1}\prod_{j=i+\ell}^{i+k-1}[t_{i,j},t_{j,i+k+\ell-1}^\varphi]^{n(i,j,k)}\cdot\tilde
w,
\end{equation*}
where $n(i,j,k)\in\f$ and $\tilde w\in M_0^*(G)$. We have that
\begin{equation*}
w^{\kappa^*}=\prod_{k=\ell+1}^{n-\ell}\prod_{i=1}^{n-k-\ell+1}t_{i,i+k+\ell-1}^{m(i,k)},
\end{equation*}
where $m(i,k)=\sum_{j=i+\ell}^{i+k-1}n(i,j,k)$. Since
$t_{i,i+k+\ell-1}$'s are independent generators of $G$, $w\in
M^*(G)$ if and only if $m(i,k)=0\in\f$ for all $i,k$.

Note that $m(i,\ell+1)=n(i,i+\ell,\ell+1)$. From Case I, Step 1 it
follows that
$[t_{i,i+1},t_{i+1,i+2\ell}^\varphi]^{n(i,i+\ell,\ell+1)}\in
M_0^*(G)$ for $1\leq i\leq n-2\ell$.

Now, choose arbitrary $i$ and $k\geq \ell+2$. Then
$n(i,i+k-1,k)=-\sum_{j=i+\ell}^{i+k-2}n(i,j,k)\in\f$. Since the
commutators
$[t_{i,i+\ell},t_{i+\ell,i+2\ell+1}^\varphi],\dots,[t_{i,i+k-1},t_{i+k-1,i+k+\ell-1}^\varphi]$
commute pairwise modulo $M_0^*(G)$, and
$[t_{i,i+k-1},t_{i+k-1,i+k+\ell-1}^\varphi]^{p}\in M_0^*(G)$, we
have that
\begin{align*}
\prod_{j=i+\ell}^{i+k-1}[t_{i,j},t_{j,i+k+\ell-1}^\varphi]^{n(i,j,k)}&\equiv&\prod_{j=i+\ell}^{i+k-2}([t_{i,j},t_{j,i+k+\ell-1}^\varphi][t_{i,i+k-1},t_{i+k-1,i+k+\ell-1}^\varphi]^{-1})^{n(i,j,k)}\\
&&\pmod{M_0^*(G)}.
\end{align*}

From Case I, Step 3 it follows that
$[t_{i,j},t_{j,i+k+\ell-1}^\varphi][t_{i,i+k-1},t_{i+k-1,i+k+\ell-1}^\varphi]^{-1}$
belongs to $M_0^*(G)$ for any $j$. Therefore, $w$ belongs to
$M_0^*(G)$.

{\bf Case III.} Let $G=\Gamma_{n,\ell}$ for some $\ell\geq 1$. We
will write $t_{ij}$ instead of $t_{ij}(1)\text{UT}_n^\ell(\f)$. Note
that $[t_{ij},t_{jk}]=t_{ik}=1\in G$ for $k\geq i+\ell$. Therefore
the group $[G,G^\varphi]$ is generated modulo $M_0^*(G)$ by the
commutators $[t_{ij},t_{jk}^\varphi]$ for $1\leq i<j<k\leq
i+\ell-1$. As we have already observed, $\Gamma_{n,1}=\{1\}$ and
$\Gamma_{n,2}$ is abelian.

Assume that $\ell\geq 3$. Following Step 2 of Case I, we can show
that every element $w\in [G,G^\varphi]$ can be written as
\begin{equation*}
w=\prod_{k=2}^{\ell-1}\prod_{i=1}^{n-k}\prod_{j=i+1}^{i+k-1}[t_{i,j},t_{j,i+k}^\varphi]^{n(i,j,k)}\cdot\tilde
w,
\end{equation*}
where $n(i,j,k)\in\f$ and $\tilde w\in M_0^*(G)$. We have that
\begin{equation*}
w^{\kappa^*}=\prod_{k=2}^{\ell-1}\prod_{i=1}^{n-k}t_{i,i+k}^{m(i,k)},
\end{equation*}
where $m(i,k)=\sum_{j=i+1}^{i+k-1}n(i,j,k)$. Since $t_{i,i+k}$'s are
independent generators of $G$, $w\in M^*(G)$ if and only if
$m(i,k)=0\in\f$ for all $i,k$. Note that $m(i,2)=n(i,i+1,2)$. From
Step 1 of Case I,  it follows that
$[t_{i,i+1},t_{i+1,i+2}^\varphi]^{n(i,i+1,2)}\in M_0^*(G)$ for
$1\leq i\leq n-2$.

Now, choose arbitrary $i$ and $k\geq 3$. Then
$n(i,i+k-1,k)=-\sum_{j=i+1}^{i+k-2}n(i,j,k)\in\f$. Since the
commutators
$[t_{i,i+1},t_{i+1,i+3}^\varphi],\dots,[t_{i,i+k-1},t_{i+k-1,i+k}^\varphi]$
commute pairwise modulo $M_0^*(G)$, and
$[t_{i,i+k-1},t_{i+k-1,i+k}^\varphi]^{p}\in M_0^*(G)$, we have that
\begin{equation*}
\prod_{j=i+1}^{i+k-1}[t_{i,j},t_{j,i+k}^\varphi]^{n(i,j,k)}\equiv\prod_{j=i+1}^{i+k-2}([t_{i,j},t_{j,i+k}^\varphi][t_{i,i+k-1},t_{i+k-1,i+k}^\varphi]^{-1})^{n(i,j,k)}\pmod{M_0^*(G)}.
\end{equation*}

From Case I, Step 3 it follows that
$[t_{i,j},t_{j,i+k}^\varphi][t_{i,i+k-1},t_{i+k-1,i+k}^\varphi]^{-1}$
belongs to $M_0^*(G)$ for any $j$. Therefore, $w$ belongs to
$M_0^*(G)$.
\end{proof}

Another open problem mentioned in \cite{KK} is whether the Bogomolov
multiplier of a central product of two groups is trivial, where both
groups have trivial  Bogomolov multipliers. Let $G_1$ and $G_2$ be
two groups with a common central subgroup (up to an isomorphism).
Namely, let $\theta:K_1\to K_2$ be an isomorphism, where $K_1\leq
Z(G_1)$ and $K_2\leq Z(G_2)$, and let $E=G_1\times G_2$. Then the
central product $G_1*G_2$ of $G_1$ and $G_2$ is defined as the
quotient group $G=E/N$, where $N=\{ab: a\in K_1,b\in K_2,
\theta(a)=b^{-1}\}\in Z(E)$. We show the triviality of the Bogomolov
multiplier for central products of unitriangular groups in the
following.

\begin{cor}\label{cor}
Let the group $G$ be a central product of two unitriangular groups
$G_1\simeq \text{UT}_k(\f)$ and $G_2\simeq \ut$, where $m\geq
1,n\geq 2,k\geq 2$. Then $B_0(G)=0$.
\end{cor}
\begin{proof}
{\bf Case I.} Let $G=G_1*G_2$, where
$G_1=\text{UT}_{n-1}(\f)=\langle t_{ij}(\lambda)\mid 1\leq i<j\leq
n-1,\lambda\in\f\rangle$ and $G_2=\ut=\langle s_{ij}(\lambda)\mid
1\leq i<j\leq n,\lambda\in\f\rangle$. Put $K_1=\langle
t_{1,n-1}(1)\rangle=Z(G_1)\simeq\f$ and $K_2=\langle
s_{1n}(1)\rangle=Z(G_2)\simeq\f$.

From Theorem \ref{main} it follows that
$B_0(G_1)=B_0(G_2)=B_0(G_1/K_1)=0$, where
$G_1/K_1=\text{UT}_{n-1}(\f)/\text{UT}_{n-1}^{n-2}(\f)=\Gamma_{n-1,n-2}$.
Thus, according to Theorem \ref{cp}, we need to construct a
homomorphism $\theta:G_1\to G_2$ such that $\theta\vert_{K_1}:K_1\to
K_2$ is an isomorphism.

Define a map $\theta:G_1\to G_2$ by
\begin{equation*}
\theta(t_{ij}(\alpha))=\left\{
\begin{array}{ll}
s_{i,j+1}(\alpha), &\ \text{if}\ i=1,\\
s_{i+1,j+1}(\alpha), &\ \text{if}\ i\geq 2,
\end{array}\right.
\end{equation*}
where $\alpha\in\f,1\leq i<j\leq n-1$. Define also
$\theta(\prod_{i,j,\alpha}t_{ij}(\alpha))=\prod_{i,j,\alpha}\theta(t_{ij}(\alpha))$.
We are going to verify that $\theta$ is a well defined homomorphism.

Indeed, we have that
$\theta([t_{1j}(\alpha),t_{jk}(\beta)])=[s_{1,j+1}(\alpha),s_{j+1,k+1}(\beta)]=s_{1,k+1}(\alpha\beta)=\theta(t_{1k}(\alpha\beta))$
and
$\theta([t_{ij}(\alpha),t_{jk}(\beta)])=[s_{i+1,j+1}(\alpha),s_{j+1,k+1}(\beta)]=s_{i+1,k+1}(\alpha\beta)=\theta(t_{ik}(\alpha\beta))$
for $i\geq 2$. Also,
$\theta([t_{1j}(\alpha),t_{km}(\beta)])=[s_{1,j+1}(\alpha),s_{k+1,m+1}(\beta)]=1=\theta(1)$
for $k\ne j$ and
$\theta([t_{ij}(\alpha),t_{km}(\beta)])=[s_{i+1,j+1}(\alpha),s_{k+1,m+1}(\beta)]=1=\theta(1)$
for $k\ne j$ and $i\geq 2$.

Finally, observe that $\theta(t_{1,n-1})=s_{1,n}$, i.e.
$\theta\vert_{K_1}:K_1\to K_2$ is an isomorphism. From Theorem
\ref{cp} it follows that $B_0(G)=0$.

{\bf Case II.} Let $G=G_1*G_2$, where $G_1=\text{UT}_k(\f)=\langle
t_{ij}(\lambda)\mid 1\leq i<j\leq k,\lambda\in\f\rangle$ and
$G_2=\ut=\langle s_{ij}(\lambda)\mid 1\leq i<j\leq
n,\lambda\in\f\rangle$. Put $K_1=\langle
t_{1k}(1)\rangle=Z(G_1)\simeq\f$ and $K_2=\langle
s_{1n}(1)\rangle=Z(G_2)\simeq\f$.

Assume that $k\leq n$. According to Case I, we can construct a chain
of homomorphisms
\begin{equation*}
\text{UT}_k(\f)\overset{\theta_1}{\longrightarrow}\text{UT}_{k+1}(\f)\overset{\theta_2}{\longrightarrow}\cdots\overset{\theta_{n-k}}{\longrightarrow}\ut.
\end{equation*}
Their composite gives a homomorphism $\theta:G_1\to G_2$ such that
$\theta\vert_{K_1}:K_1\to K_2$ is an isomorphism. Apply Theorem
\ref{cp}


\end{proof}

\end{document}